\documentclass[letterpaper,10 pt, conference]{ieeeconf}
\IEEEoverridecommandlockouts

\usepackage{amsmath, amssymb}
\usepackage{mathtools}
\usepackage{pgfplots}
\usepackage{tikz}

\DeclarePairedDelimiter\lr{\lparen}{\rparen}

\usetikzlibrary{arrows,automata, patterns, calc,decorations.pathmorphing,decorations.markings}

\newcommand{\bbR}{\mathbb{R}}

\newcommand{\calC}{\mathcal{C}}

\newcommand{\btab}{\begin{center}\def\arraystretch{1.5}\begin{tabular}}
        \newcommand{\etab}{\end{tabular}\end{center}}
\newcommand{\bbm}{\begin{bmatrix*}}
    \newcommand{\ebm}{\end{bmatrix*}}
\newcommand{\bvm}{\begin{vmatrix*}}
    \newcommand{\evm}{\end{vmatrix*}}

\newcommand{\inv}{^{-1}}
\renewcommand{\t}{^\top}

\newcommand{\set}[2]{\left\{ #1 ~\left|~ \vphantom{#1} #2 \right. \right\}}

\newcommand{\qand}{\quad\text{and}\quad}

\newcommand{\ddt}[1]{\tfrac{\d^{#1}}{\d t^{#1}}}
\newcommand{\pddt}{\lr*{\ddt{}}}
\renewcommand{\d}{\textup{d}}

\newcommand{\cinf}[1]{\calC^\infty_{#1}}
\newcommand{\sys}{{\Sigma}}
\newcommand{\ass}{\text{\upshape A}}
\newcommand{\env}{\text{\upshape E}}
\newcommand{\gar}{{\Gamma}}
\newcommand{\con}{\mathcal{C}}
\newcommand{\contract}[1]{\con_{#1} = (\ass_{#1}, \gar_{#1})}

\newcommand{\simby}{\preccurlyeq}

\newcommand{\meet}{\wedge}
\newcommand{\join}{\vee}

\newcommand{\B}[1]{\mathfrak{B}\lr{#1}}

\newcommand{\Bi}[1]{\mathfrak{B}_{\textup{i}}\lr{#1}}

\newcommand{\Bo}[1]{\mathfrak{B}_{\textup{o}}\lr{#1}}

\renewcommand{\epsilon}{\varepsilon}

\newtheorem{theorem}{Theorem}
\newtheorem{lemma}[theorem]{Lemma}

\newtheorem{remark}{Remark}
\newtheorem{definition}{Definition}
\newtheorem{example}{Example}

\title{Behavioural assume-guarantee contracts for linear dynamical systems}
\author{B. M. Shali, A. J. van der Schaft, B. Besselink\thanks{The authors are with the Jan C. Willems Center for Systems and Control, and the Bernoulli Institute for Mathematics, Computer Science, and Artificial Intelligence, University of Groningen, Groningen, The Netherlands; Email: {\emph{b.m.shali@rug.nl}}; \emph{a.j.van.der.schaft@rug.nl}; \emph{b.besselink@rug.nl}.}}
\begin{document}
	\maketitle

    \begin{abstract}
        Motivated by the growing requirements on the operation of complex engineering systems, we present contracts as specifications for continuous-time linear dynamical systems with inputs and outputs. A contract is defined as a pair of assumptions and guarantees, both characterized in a behavioural framework. The assumptions encapsulate the available information about the dynamic behaviour of the environment in which the system is supposed to operate, while the guarantees express the desired dynamic behaviour of the system when interconnected with relevant environments. In addition to defining contracts, we characterize contract implementation, and we find necessary conditions for the existence of an implementation. We also characterize contract refinement, which is used to characterize contract conjunction in two special cases. These concepts are then illustrated by an example of a vehicle following system.
    \end{abstract}

    \section{Introduction}

    Contemporary engineering systems are becoming increasingly more complex. They often comprise a large number of interconnected components, which can be quite complex themselves, thus making the design and analysis of the overall system prohibitively difficult. One way to circumvent this problem is by employing a method for design and analysis that is inherently \emph{modular}, i.e., that allows components to be considered independently. Contract-based design, which finds its origins in the field of software engineering \cite{meyer1992}, is precisely such a method.

    Motivated by this, we present assume-guarantee contracts for continuous-time linear dynamical systems with inputs and outputs. There are many different types of contract theories depending on the context in which they are developed \cite{benveniste2018}. Nevertheless, they all share a common objective, namely, to support the independent design of components within interconnected systems. This is typically done by defining appropriate concepts, such as refinement and composition, that satisfy certain compositionality properties. The concepts and philosophy behind contract theories are captured in a mathematical meta-theory of contracts presented in \cite{benveniste2018}.

    The contracts in this paper provide an alternative to common methods for expressing system specifications in control, such as dissipativity \cite{willems1972} and set-invariance \cite{blanchini1999}. While the latter represent specifications that are typically static, e.g., static supply rates and invariant sets, the contracts in this paper are used to specify dynamic behaviour. Furthermore, the contracts in this paper express specifications directly in the continuous domain, unlike formal methods in control \cite{tabuada2009}, which usually require the abstraction to discrete transition systems in order to express specifications using LTL.

    We make the following contributions in this paper. First, we define assume-guarantee contracts for linear systems, where the assumptions and guarantees are linear systems that represent an expected input behaviour and a desired external behaviour, respectively. These contracts allow us to express rich specifications directly in the continuous domain. Second, we define and characterize contract implementation, which leads to necessary conditions for contract consistency, i.e., the existence of an implementation. Third, we define and characterize contract refinement, which we then use to define contract conjunction and characterize it in two special cases.

    The contracts presented in this paper are a generalization of the contracts in \cite{shali2021}, and are closely related to the contracts in \cite{besselink2019}. In contrast to \cite{shali2021}, where the guarantees specify only a set of admissible output trajectories, the guarantees in this paper can also specify a particular relationship between input and output trajectories. As such, the contracts in this paper truly specify input-output behaviour and enable the expression of significantly richer specifications than \cite{shali2021}. At the same time, the developed theory is based on the equations representing linear systems, not on their solutions, which leads to conditions that are much easier to verify.

    Different types of contracts have already been used as specifications for dynamical systems. A small-gain theorem for parametric assume-guarantee contracts is presented in \cite{kim2017} and used for controller synthesis in \cite{khatib2020}. These contracts are also adopted in \cite{chen2021} for safety-critical control synthesis in network systems. On the other hand, assume-guarantee contracts that capture invariance properties are presented in \cite{saoud2018} and used for symbolic controller synthesis in \cite{saoud2018b, saoud2021}. Applications of these contracts can be found in \cite{loreto2020, zonetti2019}. Contracts for safety are presented in \cite{eqtami2019}, while the contracts in \cite{ghasemi2020} are used for finite-time reach and avoid or infinite-time invariance. Finally, the authors of \cite{sharf2021} focus on using contracts to represent linear constraints, and develop efficient computational tools based on linear programming. As opposed to the contracts in this paper, which are used to express specifications on the dynamics of continuous-time systems, the contracts in \cite{kim2017, khatib2020, chen2021, sharf2021} are defined only for discrete-time systems, and the contracts in \cite{saoud2018, saoud2018b, saoud2021, loreto2020, zonetti2019, eqtami2019} cannot express specifications on dynamics.

    The remainder of this paper is structured as follows. In Section~\ref{sec:system_class}, we introduce the class of systems that will be considered in this paper. Contracts are introduced in Section~\ref{sec:contracts}, where we define and characterize contract implementation, and discuss contract consistency. Contract refinement and conjunction are treated in Section~\ref{sec:refinement_and_conjunction}. We demonstrate the theoretical framework with an illustrative example in Section~\ref{sec:example}, followed by a conclusion in Section~\ref{sec:conclusion}.

    The notation in this paper is mostly standard. We denote the space of smooth functions from $\bbR$ to $\bbR^n$ by $\cinf{n}$. A \emph{polynomial matrix} is a matrix whose entries are polynomials, and a \emph{rational matrix} is a matrix whose entries are rational functions. Throughout this paper, all polynomials are univariate and have real coefficients. A rational function is \emph{proper} if the degree of its denominator is greater than or equal to the degree of its numerator, and a rational matrix is proper if all of its entries are proper rational functions. A square polynomial matrix $P(s)$ is called \emph{invertible} if there exists a rational matrix $Q(s)$ such that $P(s)Q(s) = I$. If there exists a \emph{polynomial} matrix $Q(s)$ such that $P(s)Q(s) = I$, then $P(s)$ is called \emph{unimodular}. In both cases, the matrix $Q(s)$ is referred to as the \emph{inverse} of $P(s)$ and is denoted by $P(s)\inv$.

    \section{Models of physical systems}\label{sec:system_class}

    In this paper, we will consider systems of the form
    \begin{equation}\label{eq:sys_iso}
        \sys: \left\lbrace
        \begin{aligned}
            \dot x &= Ax + Bu,\\
            y &= Cx + Du,
        \end{aligned}\right.
    \end{equation}
    where $x\in\cinf{n}$ is the state trajectory, $u\in\cinf{m}$ is the input trajectory, and $y\in\cinf{p}$ is the output trajectory. We regard $u$ and $y$ as external variables that can interact with the environment, while $x$ is internal, as illustrated in Figure~\ref{fig:sys}.
    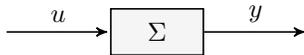
\begin{figure}[!htpb]
        \centering
        \begin{tikzpicture}[->,>=stealth',shorten >=1pt,auto,node distance=3cm,
            semithick]
            \tikzset{box/.style = {shape = rectangle,
                    color=black,
                    fill=white!96!black,
                    text = black,
                    inner sep = 5pt,
                    minimum width = 35pt,
                    minimum height = 17.5pt,
                    draw}
            }

            \node[box] (S2) at (2,0) {$\sys$};

            \draw (0,0) -- node[midway, above] {$u$} (S2);
            \draw (S2) -- node[midway, above] {$y$} (4,0);
        \end{tikzpicture}
        \caption{The system $\sys$ as a signal processor.}
        \label{fig:sys}
        \vspace{-3mm}
    \end{figure}

   Since we are interested in the interaction of a system with its environment, we consider the \emph{external behaviour}
    \begin{equation}
        \B{\sys} = \set{(u,y)\in\cinf{m+p}}{\exists x\in\cinf{n} \text{ s.t. }\eqref{eq:sys_iso} \text{ holds}}.
    \end{equation}
    In the behavioural approach to systems theory \cite{polderman1998, willems2007a}, the system $\sys$ is seen as a representation of its external behaviour $\B{\sys}$. In view of \cite[Theorem~6.2]{willems1983}, the same external behaviour can always be represented by a system of the form
    \begin{equation}\label{eq:sys_io}
        \sys: P\pddt y = Q\pddt u,
    \end{equation}
    where $u\in\cinf{m}$, $y\in\cinf{p}$, and $P(s)$ and $Q(s)$ are polynomial matrices such that $P(s)$ is invertible and $P(s)\inv Q(s)$ is proper. If these conditions on $P(s)$ and $Q(s)$ are satisfied, then we say that $\sys$ of the form \eqref{eq:sys_io} is in \emph{input-output form} \cite[Section~3.3]{polderman1998}. As \eqref{eq:sys_io} involves only the external variables $u$ and $y$, we will represent external behaviours by systems $\sys$ of the form \eqref{eq:sys_io} in input-output form.

    The concept of behaviour allows one to compare different systems. In particular, if $\B{\sys_1}\subset\B{\sys_2}$, then for a given input trajectory $u\in\cinf{m}$, the set of output trajectories produced by $\sys_1$ is contained in the set of output trajectories produced by $\sys_2$, hence $\sys_2$ can be interpreted as having richer dynamics than $\sys_1$. Behavioural inclusion plays a major role in the definition of a contract and its related concepts. The following theorem, whose proof can be found in \cite{shali2021}, see also \cite{polderman2000}, provides an algebraic characterization of behavioural inclusion that will be used in the following sections.
    \begin{theorem}[{\cite[Theorem~2]{shali2021}}]\label{thm:inclusion}
        Consider the behaviours
        \begin{equation*}
            \mathfrak{B}_j = \set{w\in\cinf{k}}{R_j\pddt w = 0},\ j\in\{1,2\},
        \end{equation*}
        where $R_1(s)$ and $R_2(s)$ are polynomial matrices. Then $\mathfrak{B}_1 \subset \mathfrak{B}_2$ if and only if there exists a polynomial matrix $M(s)$ such that $R_2(s) = M(s) R_1(s)$.
    \end{theorem}

    \section{Contracts}\label{sec:contracts}

    In this section, we will define contracts and show how they can be used as specifications. Consider a system $\sys$ of the form \eqref{eq:sys_io} in input-output form. As an open system, $\sys$ typically operates in interconnection with its environment, which is responsible for providing input trajectories. In view of this, we define an \emph{environment} $\env$ as a system of the form
    \begin{equation}\label{eq:env}
    	\env:  0 = E\pddt u,
    \end{equation}
    where $u\in\cinf{m}$ and $E(s)$ is a polynomial matrix. The environment $\env$ defines the \emph{input behaviour}
    \begin{equation*}
    	\Bi{\env} = \set{u\in\cinf{m}}{\eqref{eq:env} \text{ holds}},
    \end{equation*}
    and the interconnection of $\sys$ with $\env$ is obtained by setting the input generated by $\env$ as input of $\sys$. This yields the system
    \begin{equation}\label{eq:env_sys}
    	\env\meet\sys: \bbm P\pddt \\[1mm] 0 \ebm y = \bbm Q\pddt \\[1mm] E\pddt \ebm u,
    \end{equation}
    which is represented graphically in Figure~\ref{fig:interconnection}.
    \begin{figure}[!htpb]
        \centering
        \begin{tikzpicture}[->,>=stealth',shorten >=1pt,auto,node distance=3cm,
            semithick]
            \tikzset{box/.style = {shape = rectangle,
                    color=black,
                    fill=white!96!black,
                    text = black,
                    inner sep = 5pt,
                    minimum width = 35pt,
                    minimum height = 17.5pt,
                    draw}
            }

            \node at (-2,0) {};
            \node[box] (S1) at (0,0) {$\env$};
            \node[box] (S2) at (2.5,0) {$\sys$};
            \node (Y) at (4.5,0) {};

            \draw (S1) -- node[midway] (UM) {} (S2);
            \draw (S2) -- node[midway, above] {$y$} (Y);

            \draw (UM.south) -- ([yshift=15] UM.south) -- node[pos=.8, above] {$u$} ([yshift=15] Y.west);

            \draw[dashed,-] ([xshift=5, yshift=25] S2.east) -- ([xshift=5, yshift=-15] S2.east) -- ([xshift=-5, yshift=-15] S1.west) -- ([xshift=-5, yshift=25] S1.west) -- ([xshift=5, yshift=25] S2.east);

        \end{tikzpicture}
        \caption{The interconnection $\env\meet\sys$.}
        \label{fig:interconnection}
        \vspace{-3mm}
    \end{figure}
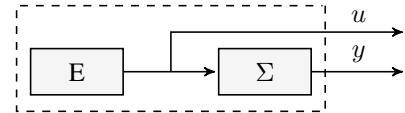

    As a design goal, we are  interested  in  guaranteeing  properties of the external behaviour $\B{\env\meet\sys}$. Since this is partially determined by the environment $\env$, any available information about $\env$ can ease the design burden on the system $\sys$ and should thus be taken into account. To formalize this, we will introduce another two systems. First, the assumptions $\ass$ are a system of the form
    \begin{equation}\label{eq:ass}
    	\ass: 0 = A\pddt u,
    \end{equation}
    where $u\in\cinf{m}$ and $A(s)$ is a polynomial matrix. Just like an environment, the assumptions $\ass$ represent the input behaviour $\Bi{\ass}$. Second, the guarantees $\gar$ are a system of the form
    \begin{equation}\label{eq:gar}
    	\gar: G\pddt y = H\pddt u,
    \end{equation}
    where $u\in\cinf{m}$, $y\in\cinf{p}$, and $G(s)$ and $H(s)$ are polynomial matrices. The guarantees $\gar$ represent the external behaviour $\B{\gar}$, just like the interconnection of a system with its environment. Then contracts are defined as follows.
    \begin{definition}\label{def:contract}
    	A \emph{contract} $\con$ is a pair $(\ass, \gar)$ of assumptions $\ass$ and guarantees $\gar$.
    \end{definition}

    The interpretation of a contract is given in the following definition.
    \begin{definition}
    	An environment $\env$ is \emph{compatible} with the contract $\contract{}$ if $\Bi{\env}\subset\Bi{\ass}$. A system $\sys$ of the form \eqref{eq:sys_io} in input-output form \emph{implements} $\con$ if $\B{\env\meet\sys}\subset\B{\gar}$ for all environments $\env$ compatible with $\con$. In such a case, we say that $\sys$ is an \emph{implementation} of $\con$.
    \end{definition}

    In other words, the assumptions capture the available information about the environment, thus leading to a class of compatible environments, while the guarantees represent the desired behaviour of the system when interconnected with any compatible environment, thus leading to a class of implementations. These two aspects of a contract constitute a formal specification for the external behaviour of a system.
    \begin{remark}
        The assume-guarantee contracts presented in this paper are very closely related to the assume-guarantee contracts presented in \cite{shali2021}. The main difference is in the guarantees, which represent an external behaviour in this paper, rather than an output behaviour only, as in \cite{shali2021}. This means that the contracts presented in this paper specify a \emph{relationship} between inputs and outputs, not only a set of admissible outputs. In any case, the contracts in \cite{shali2021} are a special case of the contracts here, obtained by restricting guarantees to be of the form \eqref{eq:gar} with $H(s) = 0$.
    \end{remark}

	Definition~\ref{def:contract} suggests that checking contract implementation requires the construction of all compatible environments. The following theorem shows that this is not necessary.
	\begin{theorem}\label{thm:implementation}
		A system $\sys$ as in \eqref{eq:sys_io} in input-output form implements $\contract{}$ if and only if $\B{\ass\meet\sys} \subset \B{\gar}$.
	\end{theorem}
	\begin{proof}
		Suppose that $\B{\ass\meet\sys}\subset\B{\gar}$ and let $\env$ be an environment compatible with $\con$. If $(u,y)\in\B{\env\meet\sys}$ then $u\in\Bi{\env}\subset\Bi{\ass}$ and $(u,y)\in\B{\sys}$, hence $(u,y)\in\B{\ass\meet\sys}\subset\B{\gar}$. This shows that $\B{\env\meet\sys}\subset\B{\gar}$ for all compatible environments $\env$, thus $\sys$ implements $\con$.	Conversely, suppose that $\sys$ is an implementation of $\calC$. Then $\B{\ass\meet\sys}\subset\B{\gar}$ because $\ass$ is compatible with $\con$.
	\end{proof}
	
	\begin{remark}\label{rem:implementation}
		Suppose that $\sys$, $\ass$ and $\gar$ are given by \eqref{eq:sys_io}, \eqref{eq:ass} and \eqref{eq:gar}, respectively. Using Theorem~\ref{thm:inclusion} with $w = [y\t \quad u\t]\t$ shows that $\B{\ass\meet\sys}\subset\B{\gar}$ if and only if there exist polynomial matrices $M_1(s)$ and $M_2(s)$ such that
		\begin{equation*}
			\bbm G(s) & -H(s)\ebm = \bbm M_1(s) & M_2(s) \ebm\bbm P(s) & -Q(s)\\ 0 & -A(s) \ebm.
		\end{equation*}
		The existence of such polynomial matrices $M_1(s)$ and $M_2(s)$ can be verified using \cite[Lemma~1]{shali2021}.
	\end{remark}

    In practice, we are only interested in contracts that can be implemented, hence the following definition.
    \begin{definition}
        A contract $\con$ is \emph{consistent} if there exists at least one implementation of $\con$.
    \end{definition}

    Not all contracts are consistent. One reason is that implementations cannot restrict the input $u$, hence any restrictions on $u$ imposed by the guarantees must already be present in the assumptions. To formalize this, we introduce the following definition.
    
    \begin{definition}\label{def:input_behaviour}
    	Given a system that involves both $u$ and $y$, its \emph{input behaviour} is defined as the projection of its external behaviour onto $u$. For guarantees $\gar$, this means that
    	\begin{equation}
    		\Bi{\gar} = \set{u\in\cinf{m}}{\exists y\in\cinf{p} \text{ s.t. } (u,y)\in\B{\gar}}.
    		\vspace{2mm}
    	\end{equation}
    \end{definition}
    
	
	Using Definition~\ref{def:input_behaviour} we obtain the following necessary condition for consistency.
    \begin{lemma}\label{lem:consistency}
        The contract $\contract{}$ is consistent only if $\Bi{\ass}\subset\Bi{\gar}$. In this case, there exist guarantees
        \begin{equation}\label{eq:gar'}
            \gar': G'\pddt y = H'\pddt u,
        \end{equation}
        such that $G'(s)$ has full row rank and $\B{\ass\meet\gar} = \B{\ass\meet\gar'}$.
    \end{lemma}
	\begin{proof}
		Suppose that $\contract{}$ is consistent. In view of Theorem~\ref{thm:implementation}, this implies that there exists $\sys$ as in \eqref{eq:sys_io} in input-output form such that $\B{\ass\meet\sys}\subset\B{\gar}$. Since $\sys$ is in input-output form, it follows that $u$ is free in $\B{\sys}$, i.e., for all $u\in\cinf{m}$, there exists $y\in\cinf{p}$ such that $(u,y)\in\B{\sys}$. Consequently, we obtain $\Bi{\ass\meet\sys} = \Bi{\ass}$, hence $\B{\ass\meet\sys}\subset\B{\gar}$ only if $\Bi{\ass}\subset\Bi{\gar}$.
		
		We proceed with finding guarantees $\gar'$ of the form \eqref{eq:gar'} such that $G'(s)$ has full row rank and $\B{\ass\meet\gar'} = \B{\ass\meet\gar}$. Let $\gar$ be given by \eqref{eq:gar}. If $G(s)$ has full row rank, then we can take $\gar' = \gar$. Now, suppose that $G(s)$ does not have full row rank. In light of \cite[Theorem~6.2.6]{polderman1998}, there exists a unimodular matrix $U(s)$ such that
		\begin{equation}\label{eq:G'_H'}
			U(s)G(s) = \bbm G'(s) \\ 0 \ebm,\quad U(s)H(s) = \bbm H'(s) \\ H''(s) \ebm,
		\end{equation}
		and $G'(s)$ has full row rank. From the same theorem, we also know that the input behaviour $\Bi{\gar}$ is given by
		\begin{equation*}
			\Bi{\gar} = \set{u\in\cinf{m}}{H''\pddt u = 0}.
		\end{equation*}
		Since $\Bi{\ass}\subset\Bi{\gar}$, Theorem~\ref{thm:inclusion} implies that there exists a polynomial matrix $M(s)$ such that $H''(s) = M(s)A(s)$. We claim that the guarantees $\gar'$ given by \eqref{eq:gar'}, where $G'(s)$ and $H'(s)$ are given by \eqref{eq:G'_H'}, are such that $\B{\ass\meet\gar} = \B{\ass\meet\gar'}$. To see this, note that $(u,y)\in\B{\ass\meet\gar}$ if and only if
		\begin{equation}\label{eq:U_gar}
			\bbm G'\pddt \\ 0 \\ 0\ebm y = \bbm H'\pddt \\ H''\pddt \\ A\pddt \ebm u
		\end{equation}
		because $U(s)$ is unimodular and \eqref{eq:G'_H'} holds, see \cite[Theorem~2.5.4]{polderman1998}. But we have that $H''(s) = M(s) A(s)$, hence \eqref{eq:U_gar} holds if and only if
		\begin{equation}\label{eq:U_gar'}
			\bbm G'\pddt \\ 0\ebm y = \bbm H'\pddt \\ A\pddt \ebm u.
		\end{equation}
		Since \eqref{eq:U_gar'} holds if and only if $(u,y)\in\B{\ass\meet\gar'}$, we conclude that $\B{\ass\meet\gar} = \B{\ass\meet\gar'}$, as desired.
	\end{proof}


    The condition that $G'(s)$ has full row rank ensures that $\gar'$ does not restrict $u$, see the proof of \cite[Theorem~6.2.6]{polderman1998}. In view of Theorem~\ref{thm:implementation} and the fact that $\B{\ass\meet\sys}\subset\B{\gar}$ if and only if $\B{\ass\meet\sys}\subset\B{\ass\meet\gar}$, it follows that the contracts $\con=(\ass,\gar)$ and $\con' = (\ass,\gar')$ define the same class of implementations. In other words, Lemma~\ref{lem:consistency} tells us that the guarantees of a consistent contract can be replaced by guarantees that do not restrict $u$.

    Unfortunately, the inclusion $\Bi{\ass}\subset\Bi{\gar}$ is not sufficient for the contract $\contract{}$ to be consistent, as illustrated by the following example.

    \begin{example}\label{ex:consistency}
        Consider assumptions $\ass$ given by \eqref{eq:ass} with $A(s) = 0$, and guarantees $\gar$ be given by \eqref{eq:gar} with $G(s) =~1$ and $H(s) = s$. It is easily seen that $\Bi{\ass} = \cinf{1} = \Bi{\gar}$, and in particular, that $\Bi{\ass}\subset \Bi{\gar}$. Nevertheless, we will show that $\contract{}$ is not consistent. Let the system $\sys$ be of the form \eqref{eq:sys_io} in input-output form. In view of Remark~\ref{rem:implementation}, we have that $\sys$ implements $\con$ if and only if there exists a polynomial matrices $M_1(s)$ and $M_2(s)$ such that
        \begin{equation*}
            \bbm 1 & -s \ebm = \bbm M_1(s) & M_2(s) \ebm \bbm P(s) & -Q(s) \\ 0 & 0\ebm.
        \end{equation*}
        But the latter holds only if $M_1(s) \neq 0$ and $P(s)\inv Q(s) = s$, which contradicts the assumption that $\sys$ is in input-output form. It follows that $\con$ does not have an implementation and thus it is not consistent.
        
        At first sight, it might seem that the inconsistency of $\con$ is caused by the fact that the guarantees $\gar$ are not in input-output form, or more precisely, that $G(s)\inv H(s)$ is not proper. This is only partially true. While it is true that  $\con$ would be consistent if the guarantees $\gar$ were in input-output form ($\sys = \gar$ would be an implementation), this is generally not necessary. Indeed, if the assumptions $\ass$ were given by \eqref{eq:ass} with $A(s) = s$ and the guarantees $\gar$ were the same, then $\con$ would be consistent because the system $\sys$ given by \eqref{eq:sys_io} with $P(s) = 1$ and $Q(s) = 0$ would be an implementation. The latter follows from Remark~\ref{rem:implementation} and the fact that
       	\begin{equation*}
       		\bbm 1 & -s\ebm = \bbm 1 & 1\ebm\bbm  1 & 0 \\ 0 & -s\ebm.
       	\end{equation*}
        In other words, even when $\Bi{\ass}\subset\Bi{\gar}$, consistency still depends on both the assumptions $\ass$ and the guarantees $\gar$.
    \end{example}

    \section{Refinement and conjunction}\label{sec:refinement_and_conjunction}

    A central concept in any contract theory is the concept of contract refinement. Refinement allows one to compare contracts, which has an essential role in enabling the independent design of components. Following the meta-theory in \cite{benveniste2018}, we define contract refinement as follows.
    \begin{definition}\label{def:refinement}
    	The contract $\con_1$ \emph{refines} the contract $\con_2$, denoted by $\con_1\simby\con_2$, if the following conditions hold:
        \begin{enumerate}
            \item all compatible environments of $\con_2$ are compatible environments of $\con_1$;
            \item all implementations of $\con_1$ are implementations of $\con_2$.
        \end{enumerate}
        \vspace{1mm}
    \end{definition}

    Said differently, $\con_1$ refines $\con_2$ if it specifies stricter guarantees that have to be satisfied under weaker assumptions, i.e., for a larger class of compatible environments. Then $\con_1$ can be seen as expressing a tighter specification than $\con_2$. Just like implementation, we can characterize refinement on the basis of assumptions and guarantees alone. If $\contract{1}$ and $\contract{2}$, then it is not difficult to see that the first condition in Definition~\ref{def:refinement} holds if and only if $\Bi{\ass_2}\subset\Bi{\ass_1}$. For the second condition, we need to consider two cases depending on whether $\con_1$ is consistent. If $\con_1$ is not consistent, then the second condition is vacuously satisfied, hence $\con_1$ refines $\con_2$ if and only if $\Bi{\ass_2}\subset\Bi{\ass_1}$. If $\con_1$ is consistent, then we have the following theorem, whose proof can be found in the appendix.

    \begin{theorem}\label{thm:refinement}
    	If the contract $\contract{1}$ is consistent, then it refines the contract $\contract{2}$ if and only if $\Bi{\ass_2}\subset\Bi{\ass_1}$ and $\B{\ass_2\meet\gar_1}\subset\B{\gar_2}$. In this case, the contract $\con_2$ is guaranteed to be consistent.
    \end{theorem}

    While refinement allows us to compare contracts, the concept of conjunction allows us to combine contracts. In particular, if a contract $\con$ refines another two contracts $\con_1$ and $\con_2$, then $\con$ can be interpreted as enforcing both $\con_1$'s and $\con_2$'s specification. This motivates the following definition.
   	\begin{definition}\label{def:conjunction}
   		The \emph{conjunction} of contracts $\con_1$ and $\con_2$, denoted by $\con_1\meet\con_2$, is the largest (with respect to contract refinement) contract that refines both $\con_1$ and $\con_2$.
   	\end{definition}
	\begin{remark}
		In Definition~\ref{def:conjunction}, by $\con_1\meet\con_2$ being the largest contract that refines both $\con_1$ and $\con_2$, we mean that $\con_1\meet\con_2$ refines both $\con_1$ and $\con_2$, and is refined by any contract that refines both $\con_1$ and $\con_2$. Note that such a contract does not necessarily exist.
	\end{remark}

    The definition of the conjunction $\con_1\meet\con_2$ has two aspects. On the one hand, $\con_1\meet\con_2$ is a contract that refines both $\con_1$ and $\con_2$, hence it expresses a fusion of the specifications that $\con_1$ and $\con_2$ express. On the other hand, $\con_1\meet\con_2$ is the \emph{largest} such contract, hence it expresses the \emph{least restrictive} fusion of the specifications that $\con_1$ and $\con_2$ express.

    Considering the first aspect of the conjunction $\con_1\meet\con_2$, we can always find a contract that refines both $\con_1$ and $\con_2$.
    \begin{lemma}\label{lem:conjunction}
        If  $\contract{1}$ and $\contract{2}$, then
        \begin{equation*}
            \con = (\ass_1\join\ass_2, \gar_1\meet\gar_2)
        \end{equation*}
        refines both $\con_1$ and $\con_2$, where
        \begin{equation*}
            \ass_1\join\ass_2: \bbm I & I \\ A_1\pddt & 0 \\ 0 & A_2\pddt \ebm \bbm l_1 \\ l_2 \ebm = \bbm I \\ 0 \\ 0 \ebm u,
        \end{equation*}
        is the \emph{join} of $\ass_1$ and $\ass_2$, and
        \begin{equation*}
            \gar_1\meet\gar_2: \bbm G_1 \pddt \\[1mm] G_2\pddt \ebm y = \bbm H_1\pddt \\[1mm] H_2\pddt \ebm u
        \end{equation*}
        is the \emph{meet} of $\gar_1$ and $\gar_2$.
    \end{lemma}
    \begin{proof}
        The join $\ass_1\join\ass_2$ is such that $u\in\Bi{\ass_1\join\ass_2}$ if and only if $u = l_1 + l_2$ for $l_1\in\Bi{\ass_1}$ and $l_2\in\Bi{\ass_2}$, while the meet $\gar_1\meet\gar_2$ is such that $(u,y)\in\B{\gar_1\meet\gar_2}$ if and only if $(u,y)\in\B{\gar_1}$ and $(u,y)\in\B{\gar_2}$. Therefore, we have that $\Bi{\ass_1\join\ass_2} = \Bi{\ass_1} + \B{\ass_2}$ and $\B{\gar_1\meet\gar_2} = \B{\gar_1}\cap\B{\gar_2}$. The former immediately implies that every environment compatible with $\con_1$ or $\con_2$ is also compatible with $\con$, hence the first condition for refinement is satisfied. Note that if $\con$ is not consistent, then it has no implementations and the second condition for refinement is vacuously satisfied. With this in mind, suppose that $\con$ is consistent. Let $\sys$ be an implementation and note that $\B{(\ass_1\join\ass_2)\meet\sys}\subset \B{\gar_1\meet\gar_2}$ due to Theorem~\ref{thm:implementation}. Since $\B{\ass_i\meet\sys}\subset \B{(\ass_1\join\ass_2)\meet\sys}$ and $\B{\gar_1\meet\gar_2}\subset \B{\gar_i}$ for $i\in\{1,2\}$, it follows that $\Bi{\ass_i\meet\sys}\subset \Bi{\gar_i}$, hence $\sys$ implements $\con_i$ due to Theorem~\ref{thm:implementation}. This shows that the second condition for refinement is satisfied and we conclude that $\con$ refines both $\con_1$ and $\con_2$.
    \end{proof}
    \begin{remark}
        Although we have defined $\ass_1\join\ass_2$ with the help of the \emph{latent variables} $l_1$ and $l_2$, due to \cite[Theorem~6.2.6]{polderman1998}, we can eliminate these to represent $\Bi{\ass_1\join\ass_2}$ by a system of the form \eqref{eq:ass}, like any other assumptions.
    \end{remark}

    While we can always find a contract that refines both $\con_1$ and $\con_2$, it is not clear whether a \emph{largest} such contract exists. In \cite{shali2021} it is shown that the conjunction of $\contract{1}$ and $\contract{2}$ exists when $\gar_1$ and $\gar_2$ are of the form \eqref{eq:gar} with $H(s) = 0$. In such a case, we have that
    \begin{equation}\label{eq:con1_meet_con2}
        \con_1\meet\con_2 = (\ass_1\join\ass_2, \gar_1\meet\gar_2).
    \end{equation}
    The following theorem, whose proof can be found in the appendix, shows that the conjunction  $\con_1\meet\con_2$ exists and has the same form in another two special cases.
    \begin{theorem}\label{thm:conjunction}
        The conjunction of contracts $\contract{1}$ and $\contract{2}$ is given by \eqref{eq:con1_meet_con2}
        if at least one of the following conditions holds:
        \begin{enumerate}
            \item $\Bi{\ass_1} = \Bi{\ass_2}$;
            \item $\B{(\ass_1\join\ass_2)\meet\gar_1} = \B{(\ass_1\join\ass_2)\meet\gar_2}$.\\
        \end{enumerate}
    \end{theorem}

    Note that the first condition in Theorem~\ref{thm:conjunction} is that $\con_1$ and $\con_2$ have the same assumptions, which implies that
    \begin{equation*}
        \con_1\meet\con_2 = (\ass_1,\gar_1\meet\gar_2) = (\ass_2, \gar_1\meet\gar_2).
    \end{equation*}
    On the other hand, the second condition is that $\con_1$ and $\con_2$ have the same guarantees when restricted to inputs in the assumptions of $\con_1\meet\con_2$, hence
    \begin{equation*}
        \con_1\meet\con_2 = (\ass_1\join\ass_2, \gar_1) = (\ass_1\join\ass_2, \gar_2).
    \end{equation*}

    \section{Illustrative example}\label{sec:example}

    In this section, we will show how contracts can be used as specifications in the context of a vehicle following system. To this end, consider two vehicles, the leader $\env$ and the follower $\sys$, which are shown graphically in Figure~\ref{fig:vehicle_following}.
    \begin{figure}[!htpb]
        \begin{tikzpicture}[->,>=stealth',shorten >=1pt,auto,node distance=2cm,
            semithick]
            \node (S) at (0,0) {\includegraphics[width=25mm]{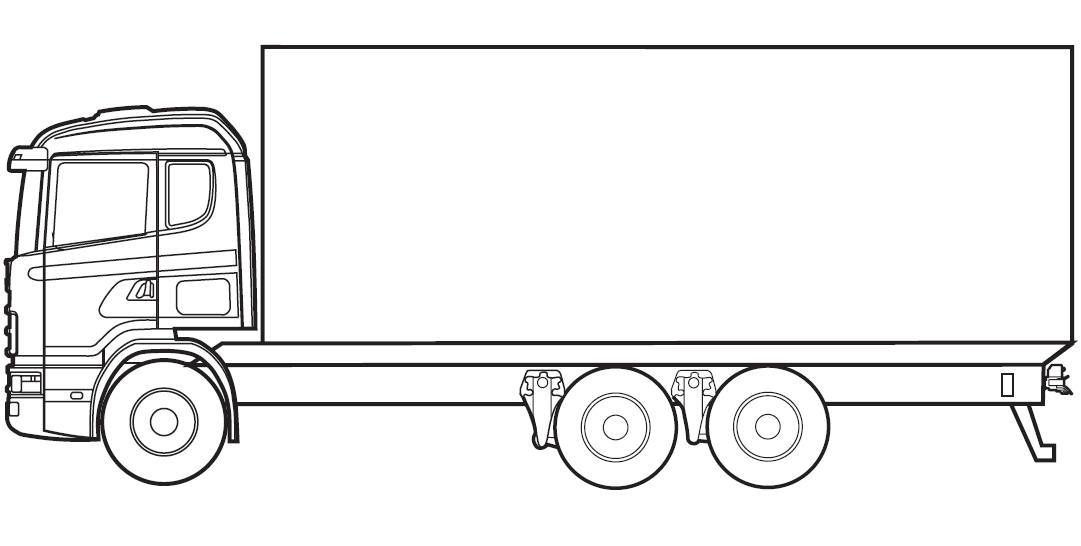}};
            \node at ([xshift=3mm, yshift=2mm] S) {$\sys$};
            \node (E) at (-3.6,0) {\includegraphics[width=25mm]{truck.png}};
            \node at ([xshift=3mm, yshift=2mm] E) {$\env$};

            \node at ([xshift=-4mm] E.west) {}; 

            \draw (E) -- node[midway, above] {$u$} (S);
            \draw (S) -- node[midway, above] {$y$} (2.25,0);
        \end{tikzpicture}
        \vspace{-2mm}
        \caption{A vehicle following system.}
        \label{fig:vehicle_following}
        \vspace{-3mm}
    \end{figure}
    We assume that the positions $p_\env$, $p_\sys$, and velocities $v_\env$, $v_\sys$, are external variables that are available for measurement. Our goal is to design $\sys$ in such a way as to guarantee tracking of the constant headway spacing policy
    \begin{equation}
        p_\env - p_\sys = hv_\sys,
    \end{equation}
    where $h > 0$ is a constant parameter. By tracking, we mean that $p_\env - p_\sys - hv_\sys$ converges to $0$ asymptotically. Note that this would be the case if
    \begin{equation}\label{eq:tracking}
        \ddt{} \left(p_\env - p_\sys - hv_\sys\right) = -k\left(p_\env - p_\sys - hv_\sys\right)
    \end{equation}
    for some constant $k > 0$.

    We can express this specification by a contract. We take the input of the follower vehicle $\sys$ to be $u = [p_\env\quad v_\env]\t$, and the output to be $y = [p_\sys\quad v_\sys]\t$. Regardless of the particular dynamics of the leading vehicle $\env$, it is safe to assume that $v_\env$ is the derivative of $p_\env$, hence the corresponding assumptions $\ass$ are given by \eqref{eq:ass} with
    \begin{equation*}
        A(s) = \bbm s & -1 \ebm.
    \end{equation*}
    On the other hand, in view of \eqref{eq:tracking}, the corresponding guarantees $\gar$ are given by \eqref{eq:gar} with
    \begin{equation*}
        G(s) = \bbm s+k & hs+ hk\ebm,\quad  H(s) = \bbm s + k & 0\ebm,
    \end{equation*}
    Then $\sys$  tracks the spacing policy (in the sense of \eqref{eq:tracking}) if $\sys$ implements the contract $\contract{}$.
	
	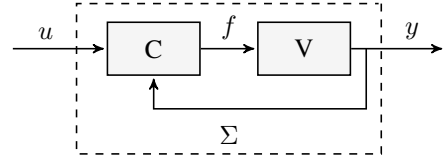
\begin{figure}
		\centering
		\begin{tikzpicture}[->,>=stealth',shorten >=1pt,auto,node distance=3cm,
			semithick]
			\tikzset{box/.style = {shape = rectangle,
					color=black,
					fill=white!96!black,
					text = black,
					inner sep = 5pt,
					minimum width = 35pt,
					minimum height = 20pt,
					draw}
			}
			\tikzset{sum/.style = {shape = circle,
					color=black,
					fill=white!96!black,
					text = black,
					minimum size=.5mm,
					inner sep=0pt,
					draw}
			}
			
			\node (u) at (-2,0) {};
			\node[box] (C) at (0,0) {C};
			\node[box] (V) at (2,0) {V};
			\node (y) at (4,0) {};
			
			\draw (C) -- node[midway, above] {$f$} (V);
			\draw (u) -- node[pos=0.35, above] {$u$} (C);
			\draw (V) -- node[pos=0.65, above] {$y$} (y);
			
			\draw ([xshift=2mm] V.east) -- ([xshift=2mm, yshift=-8mm] V.east) -- ([yshift = -8mm] C.center) -- (C);
			
			\draw[dashed, -] ([xshift=4mm, yshift=6mm] V.east) -- ([xshift=4mm, yshift=-14mm] V.east) -- node[midway, above] {$\sys$} ([xshift=-4mm, yshift=-14mm] C.west) -- ([xshift=-4mm, yshift=6mm] C.west) --  ([xshift=4mm, yshift=6mm] V.east);
			%

		\end{tikzpicture}
		\caption{The controlled vehicle $\sys$.}
		\label{fig:controlled_vehicle}
		\vspace{-5mm}
	\end{figure}

    We now turn to verifying that a given controlled vehicle $\sys$ implements $\con$. To this end, we will model the follower vehicle as a simple mechanical system with unit mass and dynamics given by
    \begin{equation*}\label{eq:vehicle_dinamics}
        \textup{V}: \bbm \ddt{} & -1 \\ 0 & \ddt{} \ebm y = \bbm 0 \\ 1\ebm f,
    \end{equation*}
    where $f$ is the force that will be provided by the controller. We claim that the controller
    \begin{equation*}
        \textup{C}:f = \bbm h\inv k & h\inv\ebm u - \bbm h\inv k & h\inv + k \ebm y
    \end{equation*}
    ensures that the controlled vehicle implements the contract $\contract{}$. To show this, note that the dynamics of the controlled vehicle $\sys = \textup{V}\meet \textup{C}$, shown in Figure~\ref{fig:controlled_vehicle}, are given by \eqref{eq:sys_io} with
    \begin{equation*}
        P(s) = \bbm s & -1 \\ h\inv k & s + h\inv + k \ebm,\quad  Q(s) = \bbm 0 & 0 \\ h\inv k & h\inv \ebm.
    \end{equation*}
    We can easily verify that
    \begin{equation*}
        \bbm G(s) & -H(s) \ebm = \bbm M_1(s) & M_2(s) \ebm \bbm P(s) & -Q(s) \\ 0 & -A(s) \ebm,
    \end{equation*}
    for $M_1(s) = \bbm 1 & h\ebm$ and $M_2(s) = 1$, which, in view of Remark~\ref{rem:implementation}, shows that $\sys$ implements $\con$. 
    
    Now, suppose that we have two types of leading vehicles, $\env_1$ and $\env_2$, whose dynamics are given by
    \begin{equation}\label{eq:env_i}
        \env_i: \bbm \ddt{} & -1 \\ 0 & \ddt{} + c_i \ebm u = \bbm 0 \\ 1 \ebm l, \quad i\in\{1,2\},
    \end{equation}
    where $c_1 = 0.25$ and $c_2 = 0.5$, and $l$ is a latent variable. Note that $\Bi{\env_1}\subset\Bi{\ass}$ and $\Bi{\env_2}\subset\Bi{\ass}$, hence $\env_1$ and $\env_2$ are environments compatible with $\con$. Consequently, we expect tracking of the headway spacing policy regardless of whether the controlled vehicle $\sys$ is following $\env_1$ or $\env_2$. This is indeed the case, as can be seen in Figure~\ref{fig:simulation}.

    To conclude this section, note that the output of the controlled vehicle $\sys$ represents its position and velocity, hence $\sys$ implements the contract $\con' = (\ass,\gar')$ with guarantees
    \begin{equation*}
        \gar': \bbm \ddt{} & -1 \ebm y = 0.
    \end{equation*}
    Furthermore, $\sys$ is easily seen to implement the conjunction $\con\meet\con' = (\ass, \gar\meet\gar')$, which we obtain by using Theorem~\ref{thm:conjunction}. Since $\Bo{\gar\meet\gar'} \subset \Bo{\gar'} = \Bi{\ass}$, an implementation $\sys$ of the conjunction $\con\meet\con'$ not only tracks the headway spacing policy, but it also guarantees that the interconnection $\env\meet\sys$, where $\env$ is an environment compatible with $\con\meet\con'$, can be interpreted as a compatible environment of $\con\meet\con'$. In other words, by defining appropriate contracts, we can ensure that several interconnected vehicles simultaneously track a given constant headway spacing policy.

    \begin{figure}
        \centering
        \input{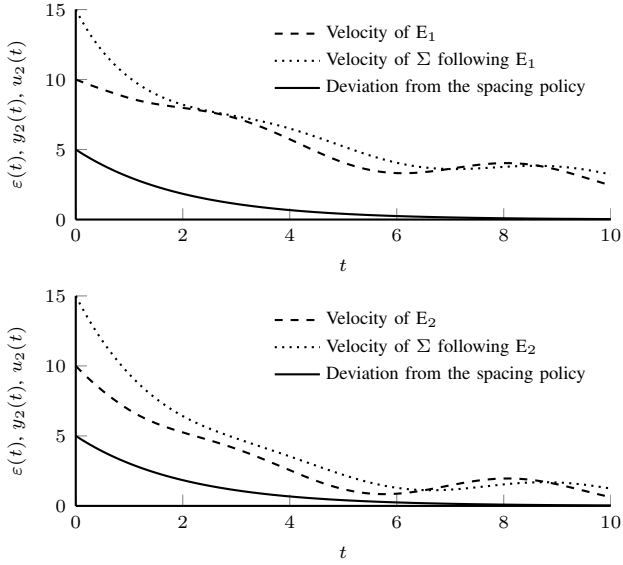}
        \vspace{-2mm}
        \caption{For each $i\in\{1,2\}$, the velocity $u_2$ of the leading vehicle $\env_i$, the velocity $y_2$ of the controlled vehicle $\sys$ following $\env_i$, and the deviation from the headway spacing policy $\epsilon = y_1 - u_1 + hy_2$, with $l$ in \eqref{eq:env_i} given by $l(t) = 1$ if $t\in[0,1]$ and $l(t) = \tfrac{1}{2} + \sin\left(\tfrac{3\pi}{8}t + \tfrac{3\pi}{16}\right)$ if $t>1$.}
        \label{fig:simulation}
        \vspace{-5mm}
    \end{figure}

    \section{Conclusion}\label{sec:conclusion}

    We presented assume-guarantee contracts for linear dynamical systems. These contracts consist of assumptions and guarantees, both of which are linear systems themselves, and which define a class of compatible environments and a class of implementations, respectively. We then characterized contract implementation on the basis of the assumptions and guarantees alone. This allowed us to derive a necessary condition for contract consistency, which we used to characterize contract refinement. Finally, through the notion of refinement, we also defined contract conjunction and characterized it in two special cases. Future work will focus on extending the theory to include appropriate notions of contract composition, which is crucial in enabling the modular design and analysis of interconnected systems.

    \appendix
    \vspace{1mm}
    \emph{Proof of Theorem~\ref{thm:refinement}:} Suppose that $\Bi{\ass_2}\subset\Bi{\ass_1}$ and $\B{\ass_2\meet\gar_1}\subset\B{\gar_2}$. Clearly, the inclusion $\Bi{\ass_2}\subset\Bi{\ass_1}$ implies that every environment compatible with $\con_2$ is also compatible with $\con_1$, such that the first condition in Definition~\ref{def:refinement} holds. To show that the second condition also holds, let $\sys$ be an implementation of $\con_1$. This implies that $\B{\ass_2\meet\sys}\subset\B{\gar_1}$ since $\ass_2$ is compatible with $\con_1$. But $\B{\ass_2\meet\sys}\subset\B{\gar_1}$ only if $\B{\ass_2\meet\sys}\subset\B{\ass_2\meet\gar_1}$, and $\B{\ass_2\meet\gar_1}\subset\B{\gar_2}$, hence $\B{\ass_2\meet\sys}\subset \B{\gar_2}$. From Theorem~\ref{thm:implementation} it follows that $\sys$ implements $\con_2$, which shows that every implementation of $\con_1$ is also an implementation of $\con_2$, i.e., the second condition in Definition~\ref{def:refinement} holds as well.

    Conversely, suppose that $\con_1$ refines $\con_2$. Since $\ass_2$ is an environment compatible with $\con_2$, it also needs to be compatible with $\con_1$, hence $\Bi{\ass_2}\subset\Bi{\ass_1}$. The only thing left to show is that $\B{\ass_2\meet\gar_1}\subset\B{\gar_2}$. In view of Theorem~\ref{thm:implementation} and the assumption that $\con_1$ is consistent, there exists a system $\sys$ as in \eqref{eq:sys_io} in input-output form such that $\B{\ass_1\meet\sys}\subset\B{\gar_1}$. Note that the latter holds if and only if $\B{\ass_1\meet\sys}\subset\B{\ass_1\meet\gar_1}$. Furthermore, due to Lemma~\ref{lem:consistency}, we know there there exist guarantees $\gar_1'$ of the form \eqref{eq:gar'} with $G'(s)$ that has full row rank such that $\B{\ass_1\meet\gar_1} = \B{\ass_1\meet\gar_1'}$, from which it follows that
    \begin{equation*}
        \B{\ass_1\meet\sys}\subset\B{\ass_1\meet\gar_1'}\subset\B{\gar_1'}.
    \end{equation*}
    Due to Theorem~\ref{thm:inclusion}, the latter holds if and only if there exist polynomial matrices $M_1(s)$ and $M_2(s)$ satisfying
    \begin{equation*}
        \bbm G_1'(s) & -H_1'(s) \ebm = \bbm M_1(s) & M_2(s) \ebm\bbm P(s) & -Q(s) \\ 0 & -A_1(s) \ebm
    \end{equation*}
    Note that $M_1(s)$ has full row rank because $G_1'(s)$ has full row rank and $P(s)$ is invertible. This implies that there exists a polynomial matrix $M'_{1}(s)$ for which $[M_1(s)\t\quad M_1'(s)\t]$ is invertible. Let $k$ be a positive integer and define
    \begin{equation*}
        P_k(s) = \bbm M_1(s) \\ s^kM_{1}'(s) \ebm P(s),\quad Q_k(s) = \bbm M_1(s) \\ s^kM_{1}'(s) \ebm Q(s).
    \end{equation*}
    We will show that the system
    \begin{equation*}
        \sys_k: P_k\pddt y = Q_k\pddt u
    \end{equation*}
    implements $\con_1$ for any positive integer $k$. Note that $P_k(s)$ is invertible, $P_k(s)\inv Q_k(s) = P(s)\inv Q(s)$ is proper, and
    \begin{equation*}
        \bbm G_1'(s) & -H_1'(s) \ebm = \bbm[r] [I \quad 0] & M_2(s) \ebm\bbm P_k(s) & -Q_k(s) \\ 0 & -A_1(s) \ebm,
    \end{equation*}
    that is, $\sys_k$ is in input-output form and $\B{\ass_1\meet\sys_k}\subset \B{\gar_1'}$. As the latter holds only if $\B{\ass_1\meet\sys_k}\subset\B{\ass_1\meet\gar_1'}$, and $\B{\ass_1\meet\gar_1'} = \B{\ass_1\meet\gar_1}$, it follows that
    \begin{equation*}
        \B{\ass_1\meet\sys_k}\subset\B{\ass_1\meet\gar_1}\subset\B{\gar_1},
    \end{equation*}
    hence $\sys_k$ implements $\con_1$ due to Theorem~\ref{thm:implementation}. But then $\sys_k$ must also implement $\con_2$, which is the case if and only if $\B{\ass_2\meet\sys_k}\subset \B{\gar_2}$. The latter holds if and only if there exist polynomial matrices $N_{1,k}(s)$ and $N_{2,k}(s)$ such that
    \begin{equation*}
        \bbm G_2(s) \hspace{-1mm}& -H_2(s) \ebm  = \bbm N_{1,k}(s) \hspace{-1mm}& N_{2,k}(s)\ebm \bbm P_k(s) \hspace{-1mm}& -Q_k(s) \\ 0 \hspace{-1mm}& -A_2(s)\ebm.
    \end{equation*}
    In particular, we have that
    \begin{equation*}
        G_2(s) = N_{1,k}(s)P_k(s) = N_{1,k}(s) \bbm M_1(s) \\ s^k M_1'(s) \ebm P(s).
    \end{equation*}
    which implies that
    \begin{equation}\label{eq:arbitrary_k}
        G_2(s)P(s)\inv \bbm M_1(s) \\ M_1'(s) \ebm\inv =  N_{1,k}(s) \bbm I & 0 \\ 0 & s^k I \ebm,
    \end{equation}
    As the left hand side of \eqref{eq:arbitrary_k} is independent of $k$, we must have that the right-hand side is independent as well. Therefore, if we partition $N_{1,k}(s) = \bbm N_{11,k}(s) & N_{12,k}(s) \ebm$, then there exist polynomial matrices $N_{11}(s)$ and $N_{22}(s)$ such that $N_{11,k}(s) = N_{11}(s)$ and $s^kN_{12,k}(s) = N_{12}(s)$ for all positive integers $k$. This is possible only if $N_{12,k}(s) = 0$ for all positive integers $k$, which implies that
    \begin{equation*}
        G_2(s) = N_{11}(s) M_1(s) P(s) = N_{11}(s) G_1'(s),
    \end{equation*}
    and, similarly, that
    \begin{align*}
        H_2(s) &= N_{11}(s) M_1(s) Q(s) + N_{2,k}(s)A_2(s) \\
        &= N_{11}(s)(H_1'(s) - M_2(s)A_1(s)) + N_{2,k}(s)A_2(s).
    \end{align*}
    Since $\Bi{\ass_2}\subset \Bi{\ass_1}$, there exists a polynomial matrix $R(s)$ such that $A_1(s) = R(s)A_2(s)$, and thus
    \begin{equation*}
        \bbm G_2(s) & - H_2(s) \ebm = \bbm N_{11}(s) & N_{2}(s) \ebm \bbm G_1'(s) & -H_1'(s) \\ 0 & -A_2(s) \ebm,
    \end{equation*}
    where $N_2(s) = N_{11}(s)M_2(s)R(s) - N_{2,k}(s)$ for an arbitrary positive integer $k$. Finally, this yields $\B{\ass_2\meet\gar_1'}\subset \B{\gar_2}$, and since $\B{\ass_2\meet\gar_1'} = \B{\ass_2\meet\gar_1}$, we conclude that $\B{\ass_2\meet\gar_1}\subset \B{\gar_2}$, as desired. \hfill\QED

    \vspace{1mm}
    \emph{Proof of Theorem~\ref{thm:conjunction}:} Due to Lemma~\ref{lem:conjunction}, we already know that the contract $\con_1\meet\con_2$ given in \eqref{eq:con1_meet_con2} refines both $\con_1$ and $\con_2$. Therefore, we only need to show that $\con_1\meet\con_2$ is the largest such contract. To this end, let $\contract{}$ be a contract that refines both $\con_1$ and $\con_2$. Then every environment compatible with $\con_1$ or $\con_2$ must also be compatible with $\con$. This is the case if and only if $\Bi{\ass_1}\subset \Bi{\ass}$ and $\Bi{\ass_2}\subset \Bi{\ass}$, which is equivalent to $\Bi{\ass_1}+\Bi{\ass_2} \subset \Bi{\ass}$ because $\Bi{\ass}$ is a linear space. The latter can be written as $\Bi{\ass_1\join\ass_2}\subset \Bi{\ass}$, which implies that every environment compatible with $\con_1\meet\con_2$ is also compatible with $\con$.

    Next, we will show that every implementation of $\con$ is also an implementation of $\con_1\meet\con_2$. This is vacuously true when $\con$ is not consistent. Suppose that $\con$ is consistent. Since $\con$ refines both $\con_1$ and $\con_2$, from Theorem~\ref{thm:implementation} it follows  that
    \begin{equation}\label{eq:con_refines_con1_and_con2}
        \B{\ass_1\meet\gar}\subset \B{\gar_1} \qand \B{\ass_2\meet\gar}\subset \B{\gar_2}.
    \end{equation}
    Furthermore, since $\Bi{\ass_1\join\ass_2}\subset \Bi{\ass}$, the same theorem tells us that $\con$ refines $\con_1\meet\con_2$ if and only if
    \begin{equation}\label{eq:con_refines_con1_meet_con2}
        \B{(\ass_1\join\ass_2)\meet\gar}\subset \B{\gar_1\meet\gar_2}.
    \end{equation}
    With this in mind, we will show that \eqref{eq:con_refines_con1_and_con2} implies \eqref{eq:con_refines_con1_meet_con2} in the following two cases.
    \begin{enumerate}
        \item Suppose that $\Bi{\ass_1} = \Bi{\ass_2}$. Then
        \begin{equation}
            \Bi{\ass_1\join\ass_2} = \Bi{\ass_1} = \Bi{\ass_2},
        \end{equation}
        from which it follows that
        \begin{equation*}
            \B{\ass_1\meet\gar} = \B{\ass_2\meet\gar} = \B{(\ass_1\join\ass_2)\meet\gar},
        \end{equation*}
        and thus \eqref{eq:con_refines_con1_and_con2} implies \eqref{eq:con_refines_con1_meet_con2}.
        \item Suppose that
        \begin{equation}\label{eq:same_guarantees}
            \B{(\ass_1\join\ass_2)\meet\gar_1} = \B{(\ass_1\join\ass_2)\meet\gar_2}
        \end{equation}
        Let $i\in\{1,2\}$ and note that \eqref{eq:con_refines_con1_and_con2} holds only if
        \begin{equation}\label{eq:assi_meet_gari}
            \B{\ass_i\meet\gar}\subset\B{\ass_i\meet\gar_i}.
        \end{equation}
        Since $\Bi{\ass_i}\subset \Bi{\ass_1\join\ass_2}$, \eqref{eq:assi_meet_gari} yields
        \begin{equation}\label{eq:ass1_join_ass2_meet_gari}
            \B{\ass_i\meet\gar}\subset \B{(\ass_1\join\ass_2)\meet\gar_i}.
        \end{equation}
        Furthermore, from \eqref{eq:same_guarantees} it follows that
        \begin{equation}\label{eq:gar1_meet_gar2}
            \B{(\ass_1\join\ass_2)\meet\gar_i} \subset \B{\gar_1\meet\gar_2},
        \end{equation}
        which, together with \eqref{eq:ass1_join_ass2_meet_gari}, implies that
        \begin{equation}
            \B{\ass_i\meet\gar} \subset \B{\gar_1\meet\gar_2}
        \end{equation}
        But $\B{\gar_1\meet\gar_2}$ is a linear space, hence
        \begin{equation}\label{eq:sum_subset_gar1_meet_gar2}
            \B{\ass_1\meet\gar} + \B{\ass_2\meet\gar}\subset \B{\gar_1\meet\gar_2}.
        \end{equation}
        Now, note that \eqref{eq:con_refines_con1_meet_con2} would follow from \eqref{eq:sum_subset_gar1_meet_gar2} if
        \begin{equation}\label{eq:join_meet_gar_subset_sum}
            \B{(\ass_1\join\ass_2)\meet\gar} \subset \B{\ass_1\meet\gar} + \B{\ass_2\meet\gar}.
        \end{equation}
        To show that \eqref{eq:join_meet_gar_subset_sum} holds, let $(u,y)\in \B{(\ass_1\join\ass_2)\meet\gar}$. It follows that $u\in\Bi{\ass_1\join\ass_2}$ and $(u,y)\in\B{\gar}$, which implies that $u = u_1 + u_2$, where $u_1\in\Bi{\ass_1}$ and $u_2\in\Bi{\ass_2}$. Since $\con$ is consistent, we must have that $\Bi{\ass}\subset\Bi{\gar}$. Therefore, for all $u\in\Bi{\ass}$, there exists $y$ such that $(u,y)\in\Bi{\gar}$. Given that $\Bi{\ass_1}\subset\Bi{\ass}$, this means that there exist $y_1$ such that $(u_1,y_1)\in\B{\gar}$. Since $\B{\gar}$ is a linear space and $(u, y)\in\B{\gar}$, we get $(u_2, y - y_2)\in\B{\gar}$. But then $(u_1,y_1)\in\B{\ass_1\meet\gar}$ and $(u_2,y-y_1)\in\B{\ass_2\meet\gar}$, hence $(u,y)\in\B{\ass_1\meet\gar} + \B{\ass_2\meet\gar}$, which shows that \eqref{eq:join_meet_gar_subset_sum} holds, as desired.
    \end{enumerate}

    In both cases, we have seen that $\eqref{eq:con_refines_con1_meet_con2}$ holds, hence $\con$ refines $\con_1\meet\con_2$. Since this is the case for any $\con$ that refines both $\con_1$ and $\con_2$, we conclude that the contract $\con_1\meet\con_2$ given in \eqref{eq:con1_meet_con2} is indeed the largest contract that refines both $\con_1$ and $\con_2$. \hfil\QED

	\bibliographystyle{ieeetr}
	\bibliography{../../references/all}
\end{document}